\documentclass[11pt]{article}
\usepackage{amsmath,amsfonts,latexsym,amssymb,amsthm,verbatim}
\usepackage{url}

\def\diam#1{\langle#1\rangle}
\newcommand{\R}{\mathbb{R}}
\newcommand{\Q}{\mathbb{Q}}

\newcommand{\Z}{\mathbb{Z}}

\newcommand{\F}{\mathbb{F}}
\begin{document}
\title{Exceptional elliptic curves over quartic fields}
\author{Filip Najman}
\date{}
\maketitle
\begin{abstract}
We study the number of elliptic curves, up to isomorphism, over a fixed quartic field $K$ having a prescribed torsion group $T$ as a subgroup. Let $T=\Z/m\Z \oplus \Z/n\Z$, where $m|n$, be a torsion group such that the modular curve $X_1(m,n)$ is an elliptic curve. Let $K$ be a number field such that there is a positive and finite number of elliptic curves $E$ over $K$ having $T$ as a subgroup. We call such pairs $(T, K)$ \emph{exceptional}. It is known that there are only finitely many exceptional pairs when $K$ varies through all quadratic or cubic fields. We prove that when $K$ varies through all quartic fields, there exist infinitely many exceptional pairs when $T=\Z/14\Z$ or $\Z/15\Z$ and finitely many otherwise.
\end{abstract}
\vspace{1cm}

\section{Introduction}

The aim of this paper is to study the number of elliptic curves over a fixed quartic field having a given torsion group $T$ as a subgroup. We do not require $T$ to be the full torsion subgroup. When counting elliptic curves, we always mean up to $\overline K$-isomorphism, and we do this without mention throughout this paper.

 First we recall what is known about torsion of elliptic curves over quartic number fields. It was shown in \cite[Theorem 3.6.]{jkp} that the if $K$ varies over all quartic fields and $E$ varies over all elliptic curves over $K$ the groups that appear  infinitely often as $E(K)_{tors}$ are exactly the following:
$$\Z /N_1 \Z,\ \ N_1=1,\ldots,18,20,21,22,24,$$
$$\Z /2\Z \oplus \Z /2N_2 \Z,\ \ N_2=1,\ldots,9,$$
$$\Z /3\Z \oplus \Z /3N_3 \Z,\ \ N_3=1,2,3,$$
$$\Z /4\Z \oplus \Z /4N_4 \Z,\ \ N_4=1,2,$$
$$\Z /N_5\Z \oplus \Z /N_5\Z,\ \ N_5=5,6.$$
An infinite family of elliptic curves over quartic fields with each of these torsion groups is constructed in \cite{jky4}. It is not known whether an elliptic curve over a quartic field can have any torsion group apart from the ones listed above. Kamienny, Stein and Stoll\footnote{see Michael Stoll's notes of a talk given at ANTS IX: \url{http://www.mathe2.uni-bayreuth.de/stoll/talks/ANTS2010-1-EllTorsion.pdf}} have announced the proof that if an elliptic curve over a quartic field has a point of prime order $p$, then $p\leq 17$.

Let $K$ be a number field and let $m$ divide $n$. Denote by $Y_1(m,n)$ the affine modular curve whose $K$-rational points classify isomorphism classes of the triples $(E, P_m, P_n)$, where $E$ is an elliptic curve (over $K$) and $P_m$ and $P_n$ are torsion points (over $K$) which generate a subgroup isomorphic to $\Z/m \Z \oplus \Z/n \Z$. For simplicity, we write $Y_1(n)$ instead of $Y_1(1,n)$. Let $X_1(m,n)$ be the compactification of the curve $Y_1(m,n)$ obtained by adjoining its cusps. Note that $X_1(m,n)$ is defined over $\Q(\zeta_m)$, where $\zeta_m$ is the $m$-th primitive root of unity.

By Mazur's torsion theorem \cite{maz1,maz2}, all the possible torsion groups of an elliptic curve over $\Q$ are known. There are $15$ of them, and for all torsion groups $\Z/m\Z \oplus \Z/n\Z$, the corresponding modular curves $X_1(m,n)$ are curves of genus 0. This makes the torsion groups from Mazur's theorem uninteresting to us, since there are infinitely many curves containing any of these torsions over any number field. If $X_1(m,n)$ is of genus 0, but such that $\Z/m\Z \oplus \Z/n\Z$ does not appear as torsion over $\Q$, then it appears 0 or infinitely many times, and it is not hard to determine (by the Hasse principle) which of these cases actually happens.

On the other hand, for torsion groups $T=\Z/m \Z \oplus \Z/n \Z$, such that $X_1(m,n)$ is a curve of genus $\geq 2$,  there are only finitely many curves containing $T$ over any number field $K$. This is a trivial consequence of Faltings' theorem \cite{fal}, since $X_1(m,n)(K)$ has only finitely many points.

This leaves the, for our purposes, interesting cases when $X_1(m,n)$ is an elliptic curve. One can see for example in \cite[Theorem 2.6, Proposition 2.7]{jkp} that this is the case only for $(m,n)=(1,11),\ (1,14),\ (1,15),\ (2,10),\ (2,12),\ (3,9),\ (4,8)$ or $(6,6)$. We denote the set of these torsion groups by $S$.

Let $K$ be a number field such that there is a positive and finite number of elliptic curves $E$ over $K$ containing $T$, where $T\in S$. We call the pairs $(T, K)$ \emph{exceptional}.

Note that an exceptional pair exists, for a given torsion $T=\Z/m \Z \oplus \Z/n \Z \in S$, over a number field $K\supset\Q(\zeta_m)$ if and only if the following conditions hold:
\begin{enumerate}

\item $rank(X_1(m,n)(K))=0$.

\item At least one of the torsion points of $X_1(m,n)(K)$ is not a cusp (note that this implies $|X_1(m,n)(K)|_{tors}>|X_1(m,n)(\Q(\zeta_m))|_{tors}$).
\end{enumerate}

The main results of this paper can be summarized in the following theorem.

\newtheorem{tm}{Theorem}
\begin{tm}
\label{gltm}
\begin{enumerate}
\item Let $T=\Z/14\Z$ or $\Z/15\Z$. There exist infinitely many quartic fields $K$ over which there exists a positive finite number of elliptic curves $E$ such that $T$ is a subgroup of $E(K)$. There are also infinitely many quartic fields  $K$ over which there are infinitely many elliptic curves $E$ with $T\subset E(K)$.
\item Let $T=\Z/11\Z$, $\Z/2\Z\oplus\Z/12\Z$, $\Z/3\Z\oplus\Z/9\Z$, $\Z/4\Z\oplus\Z/8\Z$, $\Z/6\Z\oplus\Z/6\Z$ or a torsion group parameterized by a modular curve of genus 0. Over every quartic field such that there exists one elliptic curve containing $T$, there exist infinitely many.
\item There is exactly one elliptic curve containing $\Z/2\Z\oplus\Z/10\Z$ over $\Q(\sqrt{2i+1})$. Over all other quartic fields there is either 0 or infinitely many elliptic curves containing $\Z/2\Z\oplus\Z/10\Z$.
\item For all other torsion groups $T$ and for all quartic fields $K$, there are finitely many (maybe none) elliptic containing $T$ over $K$.
\end{enumerate}
\end{tm}

\textbf{Remark 1.} Note that the first case in Theorem \ref{gltm} does not happen over quadratic and cubic fields, while all the others do.

\textbf{Remark 2.} It is not hard to work out, by using the results from this paper and invoking \cite[Theorem 3.6.]{jkp}, that if one would want to count, for all possible $T$, the number of elliptic curves having $T$ as the full torsion (not just as a subgroup), that all one would have to do is to determine whether each of the groups from Mazur's theorem appears over every quartic field as the full torsion group for infinitely many curves (we certainly believe this to be true). Note that the fact that there are infinitely many elliptic curves with $T$ as its full torsion group, for every $T$ from Mazur's theorem, was proved for cubic fields in \cite{naj}.\\

Before proving the statements of Theorem \ref{gltm}, in Section \ref{sec2} we briefly recall results about exceptional pairs over quadratic and cubic fields and show that their finiteness should be of no surprise.

\section{Exceptional pairs over quadratic and cubic fields}
\label{sec2}

It is known that there are only 3 exceptional pairs over all quadratic fields. These are the pairs $(\Z/14\Z, \Q(\sqrt{-7}))$, $(\Z/15 \Z, \Q(\sqrt{5}))$ and $(\Z/15 \Z, \Q(\sqrt{-15}))$. There are exactly two elliptic curves over $\Q(\sqrt{-7})$ with torsion $\Z/14\Z$ and one curve with torsion $\Z/15 \Z$ over each of the fields $\Q(\sqrt{5})$ and $\Q(\sqrt{-15})$. The proofs and explicit models of the curves in question can be found in \cite{kn}.

Similarly, there exist only two exceptional pairs over cubic fields, $(\Z/11\Z,K_1)$ and $(\Z/14\Z,K_2)$, where the definition of $K_1$ and $K_2$ can be found in \cite{naj}. There is exactly one elliptic curve with torsion $\Z/11\Z$ over $K_1$ and two curves with torsion $\Z/14\Z$ over $K_2$. The proofs and explicit models of these curves and fields can also be found in \cite{naj}.

The finiteness of exceptional pairs over quadratic and cubic fields is not surprising, as it follows trivially from the following proposition.

\newtheorem{tmprp}[tm]{Proposition}
\begin{tmprp}
Let $A$ be a set of number fields such that there exists a bound $B$ such that every number field $K\in A$ is of degree $\leq B$ and that for all fields $K\in A$, the only subfield of $K$ is $\Q$. Then there are only finitely many exceptional pairs $(T,K)$, with $K\in A$.
\label{prop1}
\end{tmprp}
\begin{proof}
Let $\{d_1,\ldots, d_k\}$ be the set of the degrees of all fields in $A$. For every $d_i$, by the Uniform Boundness Conjecture (which is a theorem of Merel \cite{mer}), there exists a bound $N_i$ such that an elliptic curve over a field of degree $d_i$ cannot have a point of order greater than $N_i$. Let $N=max\{N_i\}$. So for any number field $K\in A$ there exists no elliptic curve with a point of order $>N$. Let $X_1(m,n)$ be an elliptic (modular) curve and $L$ the field obtained by adjoining all the $n$-torsion points of $X_1(m,n)$, for all $n\leq N$. We conclude that $X_1(m,n)$ has torsion larger over $K\in A$ than over $\Q$ only if $K$ contains a subfield $\neq \Q$ of $L$. Since the only proper subfield of $K$ is $\Q$, this implies that $K$ is a subfield of $L$. But since $L$ is a number field, there are only finitely many subfields of $L$ and we are done.
\end{proof}

 The set of all fields of degree $p$, where $p$ is a fixed prime, satisfies the assumptions of Proposition \ref{prop1}, so a trivial corollary of it is that there are only finitely many exceptional curves over all fields of degree $p$.

Note that the assumption in Proposition \ref{prop1} that all $K\in A$ have no proper subfields is necessary. If this assumption is removed, then, using the notation as in the proof of Proposition \ref{prop1}, there can exist infinitely many $K\in A$ such that $K\cup L$ is some subfield of $L$ over which $X_1(m,n)$ has torsion larger than over $\Q$. We will see precisely this phenomenon in the next section.

\section{Exceptional pairs for torsion $\Z/14 \Z$ and $\Z/15 \Z$}

In this section we prove that there are infinitely many exceptional pairs $(T,K)$ where $T=\Z/14 \Z$ or $\Z/15 \Z$ and $K$ runs through all quartic fields.

Our strategy will use the fact that there exist exceptional pairs over quadratic fields. Let $L$ be the quadratic field from such an exceptional pair. We then find infinitely many quadratic extensions $K$ of $L$, such that there is still only finitely many elliptic curves over $K$ containing $T$.

It is easy to see that this happens if and only if the corresponding modular curve $X_1(m,n)$ remains a rank 0 elliptic curve over the biquadratic field $K$.

By $E^{(d)}$ we denote the quadratic twist of $E$ by $d$. We will use the well-known fact \cite[Exercise 10.16.]{sil} that if $K$ is a number field, $L$ a quadratic extension of $K$, $L=K(\sqrt d)$, and $E$ an elliptic curve defined over $K$, then
$$rank(E(L))=rank(E(K))+rank(E^{(d)}(K)).$$

In our case if $F=\Q(\sqrt {d_1}, \sqrt {d_2})$, since the modular curves $X_1(n)$ are defined over $\Q$, we have
$$rank(X_1(n)(\Q(\sqrt {d_1}, \sqrt {d_2}))=rank(X_1(n)(\Q(\sqrt {d_1}))+$$ $$+rank(X_1^{(d_2)}(n)(\Q(\sqrt {d_1}))
=rank(X_1(n)(\Q))+$$ $$+rank(X_1^{(d_1)}(n)(\Q))+rank(X_1^{(d_2)}(n)(\Q))+rank(X_1^{(d_1d_2)}(n)(\Q)).$$

We start with $X_1(14)$. Since we know that there is a positive finite number of elliptic curves over $\Q(\sqrt{-7})$ with 14-torsion, we have to prove that there exists infinitely many quadratic extensions $F$ of $\Q(\sqrt{-7})$ such that $rank(X_1(14)(\Q(\sqrt {-7},$ $\sqrt {d}))=0$. We actually prove a stronger statement, namely that there are infinitely many primes $p$ such that $$rank(X_1(14)(\Q(\sqrt {-7},\sqrt {p}))=0.$$

As previously noted we have to prove
$$rank(X_1(14)(\Q(\sqrt {-7}, \sqrt {p})))=rank(X_1(14)(\Q))+rank(X_1^{(-7)}(14)(\Q))+$$
$$+rank(X_1^{(p)}(14)(\Q))+rank(X_1^{(-7p)}(14)(\Q))=0.$$

As we already know that $rank(X_1(14)(\Q))=rank(X_1^{(-7)}(14)(\Q))=0$, we have to find primes $p$ such that $$rank(X_1^{(p)}(14)(\Q))=rank(X_1^{(-7p)}(14)(\Q))=0.$$

We do this by examining \emph{Selmer groups}. We will do a descent using $2$-isogenies. This method is well suited for our purposes, since both $X_1(14)(\Q)$ and $X_1(15)(\Q)$ have points of order 2. A more thorough treatment and proofs associated with descent via $2$-isogenies can be found in \cite{sil}.

Suppose an elliptic curve $E$ over $\Q$ (we can specialize to $\Q$, as we are going to do descent over $\Q$) has a point of order 2. Then $E$ can be written as
$$E: y^2=x^3+ax^2+bx$$
with $a,b \in \Z$. There exists a $2$-isogeny $\phi$ from $E$ to
$$E':y^2=x^3-2ax^2+(a^2-4b)x.$$
There is also a dual isogeny $\psi: E' \rightarrow E$ such that
$\psi\circ \phi=[2]_{E}$ and $\phi\circ \psi=[2]_{E'}$. Define the map
$\alpha:E(\Q )\rightarrow \Q^*/{\Q^*}^2$ as $\alpha(O)={\Q^*}^2$,
$\alpha((0,0))=b{\Q^*}^2$ and for all other points $P=(x,y)\in E(\Q),$ $\alpha((x,y))=x{\Q^*}^2.$ In exactly the same way we define the map $\beta:E'(\Q )\rightarrow \Q^*/{\Q^*}^2$. One can show that $Im(\alpha)$  is isomorphic to $E(\Q)/\psi(E'(\Q))$ and $Im(\beta)$ is isomorphic to $E'(\Q)/\phi(E(\Q))$.

The relation
\begin{equation}rank(E(\Q))=\dim_{\F_2}(Im(\alpha))+\dim_{\F_2}(Im(\beta))-2\label{sel}\end{equation}
yields the desired rank.

It can be shown that $Im(\alpha)$ consists of classes $b_1{\Q^*}^2$, where $b_1$ is squarefree, $b=b_1b_2$, such that the homogenous space
$$C_{b_1}: a^2=b_1u^4+au^2v^2+b_2v^4$$
has a non-trivial solutions in integers. Note that one can assume $(u,v)=1$ without loss of generality. A similar statement holds for $Im(\beta)$.

The set of classes $b_1{\Q^*}^2$ such that $C_{b_1}$ is locally solvable everywhere (including at $\infty$) is called the Selmer group corresponding to the 2-isogeny $\psi$ and we denote it by $S(\psi)$.

Similarly, we denote the Selmer group corresponding to the 2-isogeny $\phi$ by $S(\phi)$.

As the Selmer groups give an upper bound for the rank,
$$rank(E(\Q))=\dim_{\F_2}Im(\alpha)+\dim_{\F_2}Im(\beta)-2 \leq\dim_{\F_2}S(\phi)+\dim_{\F_2}S(\psi)-2,$$
this gives us a method of proving that a certain elliptic curve has rank 0.

\medskip

We turn out attention back to $X_1(14)$.

\newtheorem{tm00}[tm]{Theorem}
\begin{tm00}
Let $p$ be a prime satisfying $p\equiv 3 \pmod 8$ and $\left(\frac{p}{7}\right)=1$. Then $rank(X_1^{(p)}(14)(\Q))=rank(X_1^{(-7p)}(14)(\Q))=0.$
\label{t1}
\end{tm00}

\begin{proof}
We give only a short version of the proof, as the complete proof would be quite technical and can be easily reproduced.

For simpler notation we write $E(n)$ instead of $X_1^{(n)}(14)(\Q)$ and $E'(n)$ for the curve 2-isogenous to $E(n)$. An explicit model of $X_1(14)$ can be found in \cite{yan}. We write the relevant curves in short Weierstrass form:
$$E(p):y^2=x^3-11px^2+32p^2x,$$
$$E(-7p):y^2=x^3+77px^2+1568p^2x,$$
$$E'(p):y^2=x^3+22px^2 - 7p^2x,$$
$$E'(-7p):y^2=x^3-154px^2-343p^2x.$$

Suppose $p\equiv 3 \pmod 8$ and $(\frac{p}{7})=1$ and let $\phi_1$ be the 2-isogeny from $E(p)$ to $E'(p)$ and $\psi_1$ its dual.

The Selmer groups $S(\psi_1)$ and $S(\phi_1)$ are contained in $\diam{-1,2,7,p}\subset \Q^*/(\Q^*)^2$. By examining the corresponding homogenous spaces locally and using the conditions on $p$, one can show that $S(\psi_1)=\diam{2}$. Similarly one can show that $S(\phi_1)=\diam{-7}$.

This proves that $rank(X_1^{(p)}(14)(\Q))=0.$

Let $\phi_2$ be the 2-isogeny from $E(-7p)$ to $E'(-7p)$ and $\psi_2$ its dual. $S(\psi_2)$ and $S(\phi_2)$ are again contained in $\diam{-1,2,7,p}\subset \Q^*/(\Q^*)^2$.

As before, one can easily deduce by local considerations that $S(\psi_2)=\diam{2}$ and $S(\phi_2)=\diam{-7}$ which implies $rank(X_1^{(-7p)}(14)(\Q))=0.$

\end{proof}

\newtheorem{cor1}[tm]{Corollary}
\begin{cor1}
There exists infinitely many primes $p$ such that for $K=\Q(\sqrt{-7}, \sqrt{p})$, the pair $( \Z / 14\Z, K)$ is exceptional.
\label{c1}
\end{cor1}
\begin{proof}
By Dirichlet's theorem on arithmetic progressions there are infinitely many primes satisfying the assumptions of Theorem \ref{t1}, and hence satisfying
$$rank(X_1(\Q(\sqrt{-7}, \sqrt{p})))=0.$$
\end{proof}

\newtheorem{tm2}[tm]{Theorem}
\begin{tm2}
Let $p$ be a prime satisfying $p\equiv 7 \pmod 8$, $p\equiv 2 \pmod 3$ and $\left(\frac{p}{5}\right)=-1$. Then $rank(X_1^{(p)}(15)(\Q))=rank(X_1^{(-15p)}(15)(\Q))=0$.
\end{tm2}
\begin{proof}
For simplicity we denote $X_1^{(p)}(15)(\Q)$ by $E(p)$ and the 2-isogenous curve by $E'(p)$.
 An explicit model of $X_1(15)$ can be found in \cite{yan}. The curve $E(p)$ is, when changed to short Weierstrass form, $$E(p):y^2=x^3-7px^2+16p^2x.$$
The elliptic curves $E(-15p)$, $E'(p)$ and $E'(-15p)$ are constructed as in the proof of Theorem \ref{t1}.

We denote $\phi_1$ to be the 2-isogeny from $E(p)$ to $E'(p)$, $\psi_1$ its dual, $\phi_2$ the 2-isogeny from $E(-15p)$ to $E'(-15p)$, and $\psi_2$ its dual.

One can se that $$S(\psi_1),\ S(\phi_1),\ S(\psi_2),\ S(\phi_2)\subset\diam{-1,2,3,5,p}\subset \Q^*/(\Q^*)^2.$$
By using the assumptions on $p$ and examining the corresponding homogenous spaces locally, one can show that $S(\psi_1)$ is trivial, i.e. $\dim_{\F_2}S(\psi_1)=0$ and $S(\phi_1)=\diam{-15,p}$. It follows that $rank(X_1^{(p)}(15)(\Q))=0$.

It is easy to work out that  $S(\psi_2)$ is trivial and $S(\phi_2)=\diam{-15,p}$. It follows that $rank(X_1^{(-15p)}(15)(\Q))=0$.

\end{proof}

In a similar manner as Corollary \ref{c1}, one can prove the following.

\newtheorem{cor2}[tm]{Corollary}
\begin{cor2}
There exists infinitely many primes $p$ such that for $K=\Q(\sqrt{-15}, \sqrt{p})$, the pair $( \Z / 15\Z, K)$ is exceptional.
\label{c2}
\end{cor2}

We would also like to know the number of fields such that there are infinitely many elliptic curves with a subgroup isomorphic to $T$, where $T=\Z/14\Z$ or $\Z/15\Z$. This is not hard as we can see in the following proposition.

\newtheorem{propinf}[tm]{Proposition}
\begin{propinf}
Let $T=\Z/14\Z$ or $\Z/15\Z$. There are infinitely many quartic fields $K$ such that there exists infinitely many elliptic curves over $K$ having $T$ as a subgroup.
\label{pinf}
\end{propinf}
\begin{proof}
We are done once we have proven that there are infinitely many quartic fields $K$ such that the elliptic curve $X_1(n)(K)$, where $n=14$ or $15$, has positive rank. It is easy to prove (see for example \cite[Lemma 3.4 b)]{jky4}) that the rank of an elliptic curve goes up in infinitely many quadratic extensions. Then it also goes up in infinitely many biquadratic fields.
\end{proof}

Corollaries \ref{c1} and \ref{c2} and Proposition \ref{pinf} prove part 1 of Theorem \ref{gltm}.

\section{Exceptional pairs for torsion $\Z/11 \Z$, $\Z/2\Z \oplus \Z/10\Z$ and $\Z/2\Z \oplus \Z/12\Z$}

It is not hard to prove that there is finitely many exceptional pairs $(T, K)$ when $T=\Z/11 \Z$, $\Z/2\Z \oplus \Z/10\Z$ or $\Z/2\Z \oplus \Z/12\Z$ and $K$ is a quartic field. In fact a much more general statement holds.

\newtheorem{lem2}[tm]{Proposition}
\begin{lem2}
Suppose $n$ is an integer and $T$ is a torsion group such that for every divisor $d$ of $n$ there exists no exceptional pairs $(T,F)$ where $F$ is a field of degree $d$. Then there is finitely many exceptional pairs $(T, K)$, where $K$ is a field of degree $n$.
\end{lem2}
\begin{proof} The proof is a variation of the reasoning used in the proof of Proposition \ref{prop1} (or \cite[Lemma 3.3]{jks}) and relies on the fact that the torsion of an elliptic curve over a number field $F$ gets larger in only finitely many cyclic extensions of $F$. The details are left as an exercise for the reader.
\end{proof}

We are left to determine, for each $T=\Z/11 \Z$, $\Z/2\Z \oplus \Z/10\Z$ or $\Z/2\Z \oplus \Z/12\Z$, whether there exists an exceptional pair $(T, K)$ at all.

As mentioned in the introduction, this amounts to finding all the quartic fields $K$ such that the corresponding modular curve $X_1(m,n)$ has noncuspidal torsion points and rank $0$.

To find over which quartic fields the torsion of $X_1(m,n)$ gets larger, we use division polynomials and Galois representations attached to $X_1(m,n)$.
For a definition and more information on division polynomials see \cite{was}. We denote by $\psi_n$ the $n$-th division polynomial, which satisfies that, for a point $P$ on an elliptic curve in Weierstrass form, $\psi_n(x(P))=0$ if and only if $nP=0$. Note that for even $n$ one has to work with $\psi_n/\psi_2$ to get a polynomial in only one variable.

Let $E[n]$ denote the $n$-th division group of $E$ over $\overline \Q$ and let $\Q(E[n])$ be the $n$-th division field of $E$. The Galois group $Gal(\overline \Q /\Q)$ acts on $E[n]$ and gives rise to an embedding $\phi_n: Gal(\Q(E[n])/\Q)  \hookrightarrow GL_2(\Z/n\Z)$ called a \emph{Galois representation associated to $E[n]$}. It is not hard to see that if $n$ is an odd prime and $\phi_n$ is surjective, then a field over which a point of order $n$ is defined is of degree $\geq 8$. For $n=2$, if $\phi _n$ is surjective, then a point of order 2 will be defined over a cubic field and the complete 2-torsion will be defined over a sextic field, but there will be no new 2-torsion over any quartic field.

For an elliptic curve without complex multiplication, this embedding is surjective for all but finitely many primes \cite{ser}. The primes for which this embedding is possibly not surjective can be computed in SAGE \cite{sag}.

We start with $T=\Z/11\Z$. The elliptic curve $X_1(11)$ is 11A3 in Cremona's tables \cite{cre} has an affine model
$$y^2-y=x^3-x^2.$$

Note that $X_1(11)(\Q)\simeq \Z/5\Z$, where all $5$ rational torsion points are cusps, and that $X_1(11)$ does not have complex multiplication (CM).

We prove the following result.

\newtheorem{tm11}[tm]{Theorem}
\begin{tm11}
There are no exceptional pairs $(\Z/11\Z, K)$, where $K$ is a quartic field.
\label{t11}
\end{tm11}
\begin{proof}

We compute that $X_1(11)$ has a surjective Galois representation $\phi_p$ for all primes $p \neq 5$.



The polynomial $\psi_5$ factors into a product of two factors of degree one (corresponding to the rational 5-torsion points) and one factor of degree 10. Hence there are no additional 5-torsion points over quartic fields.

It remains to check the 25-torsion. The polynomial $\psi_{25}$ has 2 degree 1 factors (corresponding to the rational 5-torsion), 2 factors of degree 5, 1 of degree 10, 2 of degree 20 and 1 of degree 250. Hence there are no quartic 25-torsion points.

We conclude that the torsion of $X_1(11)$ stays $\Z/5\Z$ over all quartic fields.
\end{proof}

Now we examine the torsion group $\Z/2\Z \oplus \Z/10\Z$. The modular curve $X_1(2,10)$ is 20A2 in \cite{cre}, is a non-CM curve and has an affine model
$$y^2=x^3+x^2-x,$$
and $X_1(2,10)(\Q)\simeq \Z/6\Z$.

\newtheorem{tm12}[tm]{Theorem}
\begin{tm12}
There is exactly one elliptic curve with torsion $\Z/2\Z \oplus \Z/10\Z$ over $\Q(\sqrt{2i+1})$. There are no exceptional pairs $(\Z/2\Z \oplus \Z/10\Z, K)$ where $K\neq \Q(\sqrt{2i+1})$ is a quartic field.
\label{t12}
\end{tm12}
\begin{proof}
First we compute that $\phi_p$ is surjective for all primes $p\neq 2,3$.

The polynomial $\psi_3$ factors as a product of a linear polynomial corresponding to the rational 3-torsion points and a degree 3 polynomial. Hence there are no additional quartic 3-torsion points.

The polynomial $\psi_9$ has all factors of degree 9 or larger if we exclude the factors of $\psi_3$, and thus there are no 9-torsion points over a quartic field.

As there is 1 point of order 2 in $X_1(2,10)(\Q)$, we know that the complete 2-torsion is defined over one quadratic field. That quadratic field is in our case $\Q(\sqrt 5)$. However, as it was proved in \cite{kn}, all the points on $X_1(2,10)(\Q(\sqrt 5))$ are cusps.

Next we factor
$$\psi_4=x(x^2+1)(x^2+x-1)(x^4 + 2x^3 - 6x^2 - 2x + 1).$$
The factor $x$ comes from the $\Q$-rational $2$-torsion, $(x^2+x-1)$ from the remaining 2-torsion points (defined over $\Q(\sqrt 5)$). The number field $F$ generated by the polynomial $x^4 + 2x^3 - 6x^2 - 2x + 1$ has $\Q(\sqrt 5)$ as a subfield, and $X_1(2,10)(F)_{tors}\simeq X_1(2,10)(\Q(\sqrt 5))_{tors}$, so we conclude that all points in $X_1(2,10)(F)_{tors}$ are cusps.

We now examine the points coming from the factor $x^2+1$. We see that the points corresponding to this factor are defined over the quartic field $K=\Q(\sqrt{2i+1})$. We find that $X_1(2,10)(K)_{tors}\simeq \Z/12\Z$ and $rank(X_1(2,10)(K))=0$. One checks that the additional torsion points are not cusps, and that all the points generate the same curve
\begin{equation}E_{\Z/2\Z \oplus \Z/10\Z}:y^2 + (w^2 + 3)xy - 8y = x^3 + (2w^2 + 2)x^2,\label{z210}\end{equation}
where $w^2=2i+1$. Note that the  $j$-invariant of this elliptic curve is 1728 and the rank is 0. Thus we have constructed an exceptional pair $(E_{\Z/2\Z \oplus \Z/10\Z},\Q(\sqrt{2i+1}))$, and we see that this is actually the only one.

\end{proof}

Note that we can easily check that $\Z/2\Z \oplus \Z/10\Z$ is the complete torsion group of the curve $E_{\Z/2\Z \oplus \Z/10\Z}$ from (\ref{z210}). Part 3 of Theorem \ref{gltm} follows from Theorem \ref{t12}.

Finally we examine the torsion group $\Z/2\Z \oplus \Z/12\Z$. The modular curve $X_1(2,12)$ is 24A4 in \cite{cre}, is a non-CM curve and has an affine model
$$y^2=x^3-x^2+x,$$
and $X_1(2,12)(\Q)\simeq \Z/4\Z$.

\newtheorem{tm13}[tm]{Theorem}
\begin{tm13}
There are no exceptional pairs $(\Z/2\Z \oplus \Z/12\Z, K)$, where $K$ is a quartic field.
\label{t13}
\end{tm13}
\begin{proof}
The strategy of the proofs is the same as in the previous two theorems.

First we compute that the Galois representation $\phi_p$ is surjective for all primes $p\neq 2$.

The curve $X_1(2,12)$ has additional 2-torsion points over $\Q(\sqrt 3)$, $\Q(\sqrt {-3})$ and $\Q(i)$, but all of them over each individual field are cusps. We check $X_1(2,12)$ over $\Q(\zeta_{12})$, the field containing the three mentioned fields as subfields, and compute $X_1(2,12)(\Q(\zeta_{12}))_{tors}\simeq \Z/2\Z \oplus \Z/8\Z$, and obtain that all the torsion points are cusps.

By checking the zeroes of $\psi_4$, one checks that there are no other quartic 4-torsion points on $X_1(2,12)$.

By checking the 8-torsion points, we find the field $K$ generated by a root of the polynomial $x^4 + 4x^3 - 6x^2 + 4x + 1$, and compute $X_1(2,12)(K)_{tors}\simeq \Z/8\Z$. But again all the torsion points are cusps.

One easily checks that there are no quartic 16-torsion points, and thus there are no exceptional pairs with torsion $\Z/2\Z \oplus \Z/12\Z$.

\end{proof}

\section{Exceptional pairs for torsion $\Z/3\Z \oplus \Z/9\Z$, $\Z/4\Z \oplus \Z/8\Z$ and $\Z/6\Z \oplus \Z/6\Z$}


Although models of modular curves parameterizing the torsion groups $T=\Z/3\Z \oplus \Z/9\Z$, $\Z/4\Z \oplus \Z/8\Z$ and $\Z/6\Z \oplus \Z/6\Z$ were found in \cite{bc}, equations for elliptic curves that arise from points on these modular curves were not given there.

Below we find the equations of the mentioned modular curves that we obtain in a slightly different way than in \cite{bc}. Also, we give the moduli spaces of elliptic curves with all the mentioned torsion groups except for $T=\Z/3\Z \oplus \Z/9\Z$. We do this both to make life easier for the reader and also because we believe that this can independently be of interest. For the case $T=\Z/3\Z \oplus \Z/9\Z$, we obtain the corresponding modular curve, and the map from the general elliptic curve $\mathcal E$ containing $T$ to $X_1(3,9)$, but not its inverse. However, in explicit examples, we were able to find elliptic curves containing $T$ from a point on $X_1(3,9)$, by computing the fiber of the map from $\mathcal E$ to $X_1(3,9)$.

As mentioned in the introduction, one can see from \cite[Proposition 2.7.]{jkp} that all three curves have genus 1.
We proceed to first construct affine models of these curves.

\newtheorem{x39}[tm]{Lemma}
\begin{x39}
The modular curve $X_1(3,9)$, defined over $\Q(\sqrt{-3})$, has an affine model
$$X_1(3,9):y^2+y=x^3.$$
\label{lem39}
\end{x39}
\begin{proof}
We start by taking a generic elliptic curve with 9-torsion over a number field $K\supset\Q(\sqrt{-3}) $ \cite[Table 3]{kub},
$$E(t):y^2+(-t^3+t^2)xy+(t^2-t+1)(t^3-t^2)y=x^3+(t^2-t+1)(t^3-t^2)x^2,$$
where $t\in K$, such that $\Delta(E(t))\neq 0$. We now factor the 3-division polynomial of this curve and obtain that for the curve to have additional points of order 3 over some field $K$ is equivalent to the equation
$$x^3 + (1/3t^6 - 2t^5 + 3t^4 - 7/3t^3 + t^2 + 1/3)x^2 + $$
\begin{equation} +(1/3t^9 - 4/3t^8 + 2t^7 - 4/3t^6 - 2/3t^5 + 2t^4 - 5/3t^3 + 2/3t^2)x +\label{jed39}\end{equation}
$$ + 1/3t^{12} -5/3t^{11} + 13/3t^{10} - 22/3t^{9} + 26/3t^8 - 22/3t^7 + 13/3t^6 - 5/3t^5 +1/3t^4=0$$
having a solution  over $K$ ($x$ in equation (\ref{jed39}) is the $x$-coordinate of the additional 3-torsion point).

We now prove that, if an elliptic curve $E$ over $K\supset\Q(\sqrt{-3})$ has a point $Q$ of order 3, and if the $x$-coordinate $u$ of a point $P=(u,v)\in E[3](\overline K), P\not\in \diam{Q}$, is in $K$, so is the $y$-coordinate $v$. Suppose the opposite, $v\not\in K$. Then $v\in L$, where $[L:K]=2$. Let $\sigma$ be the generator of $Gal(L/K)$. As $2P=-P$, it follows that $x(2P)=x(P)$ and $\sigma(P)=-P=2P$. Let $w$ be the Weil pairing on $E[3]=\diam{P,Q}$, and $w(P,Q)=\zeta$, where $\zeta$ is some primitive third root of unity. But $\sigma(\zeta)=\zeta$ and $\sigma(w(P,Q))=w(\sigma(P),\sigma(Q))=w(2P,Q)=\zeta^2$, which is a contradiction. Thus the existence of a solution of equation (\ref{jed39}) is a sufficient condition for all of $E[3]$ to be defined over $K$.

We next compute the projective model (using MAGMA \cite{mag}) of the curve (\ref{jed39})
$$Y^{12} - 5Y^{11}Z + XY^9Z^2 + 13Y^{10}Z^2 - 4XY^8Z^3 - 22Y^9Z^3 + $$
$$+X^2Y^6Z^4 + 6XY^7Z^4 + 26Y^8Z^4 - 6X^2Y^5Z^5 - 4XY^6Z^5 -$$
\begin{equation}-    22Y^7Z^5 + 9X^2Y^4Z^6 - 2XY^5Z^6 + 13Y^6Z^6 - 7X^2Y^3Z^7 + \label{proj39}\end{equation}
$$+    6XY^4Z^7 - 5Y^5Z^7 + 3X^2Y^2Z^8 - 5XY^3Z^8 + Y^4Z^8 +$$
$$ + 3X^3Z^9  + 2XY^2Z^9 + X^2Z^{10}=0,$$
and using $(1:0:0)$ as a basepoint construct the elliptic curve (again using MAGMA)
$$y^2 + 6xy - 81y = x^3 - 36x^2 + 486x - 2187,$$
which is isomorphic to
$$y^2 + y = x^3.$$
\end{proof}

Note that the elliptic curve $X_1(3,9)$ is 27A3 in \cite{cre}.

\newtheorem{x48}[tm]{Lemma}
\begin{x48}
The modular curve $X_1(4,8)$, defined over $\Q(i)$, has an affine model
$$X_1(4,8):y^2=x^3-x.$$
\label{lem48}
\end{x48}
\begin{proof}
We start by taking a generic elliptic curve containing torsion $\Z/2\Z \oplus \Z/8\Z$ over a number field $K$ \cite[Table 3]{kub}, which has the form
$$E(t):y^2=x\left(x+\left(t^2-1\right)^2\right)\left(x+\left(\frac{4t^2}{t^2-1}\right)^2\right)$$
where $t\in K$, such that $\Delta (E(t))\neq 0$. The point of order $2$ that is not by default divisible by 2 is $(-16t^4/(t^4 - 2t^2 + 1),0)$. One can see that for it to be divisible by 2 in a number field $K$ (which implies that $E(t)$ contains torsion $\Z/4\Z \oplus \Z/8\Z$), by the 2-descent Theorem \cite[Theorem 4.2., p. 85]{kna}, it is equivalent to the expression $$\left(\frac{t^2+1}{t^2-1}\right)^2(t^2 - 2t - 1)(t^2 + 2t - 1)$$
being a square. So this is equivalent to
\begin{equation}s^2=(t^2 - 2t - 1)(t^2 + 2t - 1)\label{jd48}\end{equation}
having a point such that $E(t)$ is not a singular curve. The curve (\ref{jd48}) is isomorphic to
$$X_1(4,8):y^2=x^3-x.$$
Note also that $E(t)$ has to be defined over a field containing $\Q(i)$ for it to have torsion $\Z/4\Z \oplus \Z/8\Z$.
\end{proof}

Note that the elliptic curve $X_1(4,8)$ is 32A2 in \cite{cre}.

\newtheorem{x66}[tm]{Lemma}
\begin{x66}
The modular curve $X_1(6,6)$, defined over $\Q(\sqrt{-3})$, has an affine model
$$X_1(6,6):y^2=x^3+1.$$
\label{lem66}
\end{x66}
\begin{proof}
We first start with a generic elliptic curve containing $\Z/3\Z \oplus \Z/6\Z$ over $\Q(\sqrt{-3})$ \cite[4.9., pp.47]{rab}, which has the form
$$E(t):y^2+2(9t^2-30t^2+60t-40)xy-144(3t-2)(3t^2+4)(3t^2-6t+4)(t-2)^3y=$$
$$=x^3-16(3t-2)(3t^2+4)(3t^2-6t+4),$$
where $t\in \Q(\sqrt{-3})$ such that $E(t)$ is not singular. For $E(t)$ to have complete 2-torsion it is necessary for it to have a square discriminant and this translates to
$$s^2=2t(3t^2-6t+4),$$
which is isomorphic to
$$y^2=x^3+1.$$
\end{proof}

Note that the elliptic curve $X_1(6,6)$ is 36A1 in \cite{cre}.

We can also work backwards in Lemma \ref{lem48} to obtain an explicit description of the universal elliptic curve over $X_1(4,8)$ over any number field $K$ containing $\Q(i)$.

\newtheorem{prop48}[tm]{Proposition}
\begin{prop48}
All elliptic curves containing $\Z/4\Z \oplus \Z/8\Z$ over some number field $K$ containing $\Q(i)$ are of the form
$$E(t):Y^2=X\left(X+\left(t^2-1\right)^2\right)\left(X+\left(\frac{4t^2}{t^2-1}\right)^2\right)$$
where $t=y/(x-1)$, and $(x,y)$ is a point on
$$X_1(4,8):y^2=x^3-x$$
over $K$, such that $\Delta (E(t))\neq 0.$
The cusps on $X_1(4,8)$ are exactly the points such that $\Delta (E(t))=0.$
\label{prp48}
\end{prop48}

In the same way, we can work our way backwards in Lemma \ref{lem66} and obtain an explicit description of the universal elliptic curve over $X_1(6,6)$ over any number field containing $\Q(\sqrt{-3})$.

\newtheorem{prop66}[tm]{Proposition}
\begin{prop66}
All elliptic curves containing $\Z/6\Z \oplus \Z/6\Z$ over some number field $K$ containing $\Q(\sqrt{-3})$ are of the form
$$Y^2+a(t)XY+b(t)Y=X^3+c(t)X^2,$$
where
$$a(t)=16/3t^3 - 32/3t^2 + 128/3t - 64/3,$$
$$b(t)=-4096/27t^8 + 20480/27t^7 - 4096/3t^6 + 65536/27t^5 -$$ $$-180224/27t^4 + 32768/3t^3 - 262144/27t^2 + 131072/27t,$$
$$c(t)=16/3t^3 - 32/3t^2 + 128/3t - 64/3,$$
and $t$ is the $x$-coordinate  of a noncuspidal point
on
$$X_1(6,6):y^2=x^3+1$$
over $K$.
The cusps on $X_1(6,6)$ satisfy
$$x(x-2)(x+1)(x^2-x+1)(x^2+2x+4)=0.$$
\label{prp66}
\end{prop66}

\textbf{Remark 3.} Unfortunately, we are unable to do the same and work back through Lemma \ref{lem39}, since the map used to obtain $X_1(3,9)$ from the projective curve (\ref{proj39}) is very complicated and we were unable to compute its inverse.

\textbf{Remark 4.} Note that although infinitely many elliptic curves containing $\Z/4\Z \oplus \Z/8\Z$ and $\Z/6\Z \oplus \Z/6\Z$ have already been constructed in \cite{jky4}, in Propositions \ref{prp48} and \ref{prp66} we have actually done much more; we computed \emph{all} such curves. Using \cite{jky4} one can obtain one curve with given torsion per quartic field, while Propositions \ref{prp48} and \ref{prp66} yield infinitely many.

\medskip

As $X_1(3,9)$, $X_1(4,8)$ and $X_1(6,6)$ have complex multiplication, their Galois representations $\phi_p$ are not surjective for infinitely many primes. Thus we cannot apply the methods we used in the previous section.

However, we can use the following lemma.

\newtheorem{lemcm}[tm]{Lemma}
\begin{lemcm}
\label{cml}
The modular curves $X_1(3,9)$ and $X_1(6,6)$ have no points of prime order $p>7$ over any quadratic extension of $\Q(\sqrt{-3})$ and $X_1(4,8)$ has no points of prime order $p>5$ over any quadratic extension of $\Q(i)$.
\end{lemcm}

\begin{proof}
From \cite{py} and \cite{sbg} it follows that in our setting $\phi(e)\leq 2w$, where $e$ is the exponent of the torsion group (over a quadratic extension of the CM field) and $w$ is the number of roots in the CM field.

After observing that all of the curves already have a point of order 3 or 4 over their CM fields, the statement of the lemma follows.
\end{proof}

\newtheorem{tm39}[tm]{Theorem}
\begin{tm39}
There are no exceptional pairs $(\Z/3\Z \oplus \Z/9\Z, K)$, where $K$ is a quartic field.
\label{t39}
\end{tm39}
\begin{proof}
Note that, by the Weil pairing an elliptic curve with complete 3-torsion has to be defined over some extension of $\Q(\sqrt{-3})$.  We first note that
$$X_1(3,9)(\Q(\sqrt{-3}))\simeq \Z/3\Z \oplus \Z/3\Z$$
and that all the points are cusps.

By Lemma \ref{cml}, we only need to examine the $p$-primary torsion for $p\leq 7$. As $X_1(3,9)(\Q(\sqrt{-3}))$ has no 2-torsion, it does not have any 2-torsion over any quadratic extension of $\Q(\sqrt{-3})$.

We factor $\psi_5,$ $\psi_7$ and $\psi_9$ and see that all their factors (apart from the ones in $\psi_9$ corresponding to $\psi_3$) have degree $\geq 3$.

Thus
$$X_1(3,9)(K)_{tors}=X_1(3,9)(\Q(\sqrt{-3}))_{tors},$$
for all quadratic extensions $K$ of $\Q(\sqrt{-3})$.




\end{proof}

\newtheorem{tm48}[tm]{Theorem}
\begin{tm48}
There are no exceptional pairs $(\Z/4\Z \oplus \Z/8\Z, K)$, where $K$ is a quartic field.
\label{t48}
\end{tm48}
\begin{proof}
We compute $X_1(4,8)(\Q(i))_{tors}\simeq \Z/2\Z \oplus \Z/4\Z$. By Lemma \ref{cml}, we need to check only the 2, 3 and 5-primary torsion. The polynomial $\psi_3$ is irreducible, hence there is no 3-torsion over any quadratic extension of $\Q(i)$.

The $5$-division polynomial $\psi_5$ has two degree 2 factors $t^2 + 1/5(-2i - 1)$ and $t^2 + 1/5(2i - 1)$, but we check that the torsion of $X_1(4,8)$ does not increase in the extensions of $\Q(i)$ generated by these polynomials.

We factor the $4$-division polynomial and obtain two factors of degree 2, $t^2 - 2t - 1$ and $t^2 + 2t - 1$, which generate the same extension, $\Q(i, \sqrt 2)$. We compute $X_1(4,8)(\Q(i, \sqrt 2))_{tors}\simeq \Z/4\Z \oplus \Z/4\Z$, so we have 8 additional torsion points. But by checking the curves generated by these points, we see that all the obtained curves are singular, hence all the points in $X_1(4,8)(\Q(i, \sqrt 2))_{tors}$ are cusps.

By factoring the $8$-division polynomial, we find that all the factors either correspond to $\psi_4$ or are of degree greater than 2.
\end{proof}

\newtheorem{tm66}[tm]{Theorem}
\begin{tm66}
There are no exceptional pairs $(\Z/6\Z \oplus \Z/6\Z, K)$, where $K$ is a quartic field.
\label{t66}
\end{tm66}
\begin{proof}
Note that, by the Weil pairing an elliptic curve with complete 3-torsion has to be defined over some extension of $\Q(\sqrt{-3})$.
We first note that $$X_1(6,6)(\Q(\sqrt{-3}))\simeq \Z/2\Z \oplus \Z/6\Z$$
and that all the points are cusps.
By Lemma \ref{cml}, we again need only to study the $p$-primary torsion for $p\leq 7$. As in the proofs of Theorems \ref{t39} and \ref{t48}, by examining the factorizations of $\psi_n$ for $n=3,4,5,7$ and 9 we see that
$$X_1(6,6)(K)_{tors}=X_1(6,6)(\Q(\sqrt{-3}))_{tors},$$
for all quadratic extensions $K$ of $\Q(\sqrt{-3})$.

\end{proof}

Part 3 of Theorem \ref{gltm} follows from Theorem \ref{t11}, \ref{t13}, \ref{t39}, \ref{t48} and \ref{t66}. Part 4 of Theorem \ref{gltm} is a trivial consequence of Faltings' theorem.

\textbf{Acknowledgements.}
The author was supported by the National Foundation for Science, Higher Education and Technological Development of the Republic of Croatia. Many thanks go to the anonymous referee for careful reading and many helpful comments that significantly improved the presentation of this paper.

\medskip

\small{DEPARTMENT OF MATHEMATICS, UNIVERSITY OF ZAGREB, BIJENI\v CKA CESTA 30, 10000 ZAGREB, CROATIA}\\

\emph{E-mail:} fnajman@math.hr


\begin{thebibliography}{1}

\bibitem{Baa}
H. Baaziz, \emph{Equations for the modular curve $X_1(N)$ and models of elliptic curves with torsion points}, Math. Comp. \textbf{79}  (2010), 2371--2386.

\bibitem{mag}
W.\ Bosma, J.~J.\ Cannon, C.\ Fieker, A.\ Steel (eds.), Handbook of Magma functions, Edition 2.17 (2011),

\bibitem{bc}
\'E. Brier, C. Clavier,
\emph{New families of ECM curves for Cunningham Numbers}. In: Proceedings of ANTS IX. LNCS \textbf{6197}, Springer, Heidelberg (2010), 96--109.

\bibitem{cre}
J. Cremona, Algorithms for Modular Elliptic Curves, 2nd ed. Cambridge University Press, Cambridge, 1997.

\bibitem{fal} G. Faltings, \emph{Endlichkeitss\"atze f\"ur abelsche Variet\"aten \"uber Zahlk\"orpern} Invent. Math. \textbf{73} (1983), 349--366.

\bibitem{jks}
D. Jeon, C.H. Kim, and A. Schweizer, \emph{On the torsion of elliptic curves over cubic number fields}, Acta Arith. \textbf{113} (2004), 291--301.

\bibitem{jkp}
D. Jeon, C.H. Kim, and E. Park, \emph{On the torsion of elliptic curves over quartic number fields}, J. London Math. Soc. (2) \textbf{74} (2006), 1--12.

\bibitem{jky4}
D. Jeon, C.H. Kim, Y. Lee,
\emph{Families of elliptic curves over quartic number fields with prescribed torsion subgroups}, Math. Comp. \textbf{80} (2011), 2395--2410.

\bibitem{kn}
S. Kamienny, F. Najman, \emph{Torsion groups of elliptic curves over quadratic fields}, Acta. Arith. \textbf{152} (2012), 291--305.

\bibitem{kub}
D.~S. Kubert, \emph{Universal bounds on the torsion of elliptic curves}, Proc. London. Math. Soc. \textbf{33} (1976), 193--237.

\bibitem{kna}
A. Knapp, Elliptic Curves, Princeton Univ. Press, 1992.

\bibitem{maz1}
B. Mazur, \emph{Modular curves and the Eisenstein ideal},  Inst. Hautes Études Sci. Publ. Math. \textbf{47} (1978),  33--186.

\bibitem{maz2}
B. Mazur, \emph{Rational isogenies of prime degree}, Invent. Math. \textbf{44} (1978), 129--162.

\bibitem{mer} L. Merel, \emph{Bornes pour la torsion des courbes elliptiques sur les corps de nombres} Invent. Math. \textbf{124} (1996), 437--449.

\bibitem{naj}
F. Najman, \emph{Torsion of elliptic curves over cubic fields}, J. Number Theory, \textbf{132} (2012), 26-36.

\bibitem{py}
D. Prasad and C.~S. Yogananda, \emph{Bounding the torsion in CM elliptic curves}, C. R. Math. Acad. Sci. Soc. R. Can. \textbf{23} (2001) 1--5.


\bibitem{rab}
F. P. Rabarison, \emph{Structure de torsion des courbes elliptiques sur les corps quadratiques}, Acta Arith. \textbf{144} (2010), 17--52.

\bibitem{ser}
J. P. Serre, \emph{Properti\'et\'es galoisiennes des points d'odre fini des courbes elliptiques}, Invent. Math. \textbf{15} (1972) 259--311.

\bibitem{sbg}
A. Silverberg, \emph{Torsion points on abelian varieties of CM-type}, Compositio Math. \textbf{68} (1988), 241--249.

\bibitem{sil}
J. Silverman, The arithmetic of elliptic curves, Grad. Text Math. 106, Springer, New York, 1986.

\bibitem{sag}
W.~A.\ Stein et~al., \emph{{S}age {M}athematics {S}oftware ({V}ersion
  4.7.1)}, The Sage Development Team, 2011, \url{http://www.sagemath.org}.

\bibitem{was}
L. Washington, Elliptic Curves. Number Theory and Cryptography, Discrete Mathematics and its Applications (Boca Raton), Chapman \& Hall/CRC, Boca Raton, 2003.

\bibitem{yan}
Y. Yang, \emph{Defining equations of modular curves}, Adv. Math. \textbf{204} (2006), 481.--508.

\end{thebibliography}
\end{document}